\documentclass[preprintnumbers,amsmath,11pt,amssymb,floatfix,superscriptaddress,nofootinbib]{article}
\usepackage{amsmath,amssymb,latexsym,textcomp,mathrsfs}
\usepackage[all]{xy}
\usepackage{graphicx}
\usepackage{bm,amsmath, amsthm, amssymb, amsfonts}
\usepackage{times}
\usepackage[english]{babel}
\usepackage{setspace}
\setbox0=\hbox{$+$}
\newdimen\plusheight
\plusheight=\ht0
\def\+{\;\lower\plusheight\hbox{$+$}\;}

\setbox0=\hbox{$-$}
\newdimen\minusheight
\minusheight=\ht0
\def\-{\;\lower\minusheight\hbox{$-$}\;}

\setbox0=\hbox{$\cdots$}
\newdimen\cdotsheight
\cdotsheight=\plusheight
\def\cds{\lower\cdotsheight\hbox{$\cdots$}}

\numberwithin{equation}{section}
\theoremstyle{plain}
\newtheorem{theorem}{Theorem}[section]
\newtheorem{lemma}{Lemma}[section]

\newtheorem{example}{Example}[section]
\newtheorem{definition}{Definition}[section]

\newtheorem{proposition}{Proposition}[section]
\newtheorem{remark}{Remark}[section]

\newtheorem{note}{Note}[section]

\def\mytitle#1{\setcounter{equation}{0}
\setcounter{footnote}{0}
\begin{flushleft}\Large\textbf{#1}\end{flushleft}
\vspace{0.20cm}}
\def\myname#1{\leftline{{\large #1}}\vspace{-0.13cm}}
\def\myplace#1#2{\small\begin{flushleft}\textit{#1}\\
\texttt{#2}\end{flushleft}}

\def\myclassification#1{\small\noindent
Keywords : weak $I^K$-Convergence, weak* $I^K$-Convergence, Condition AP($I,K$), weak $I^K$-Limit Points, weak* $I^K$-Limit Points.\\
AMS subject classification: Primary: 40A35; Secondary: 40H05
       #1\vspace{0.5cm}}
\usepackage{graphicx}
\usepackage{amsmath}

\begin{document}
\mytitle{Weak and weak* $I^K$-convergence in normed spaces}

\myname{$Amar Kumar Banerjee^{\dag}$\footnote{akbanerjee1971@gmail.com, akbanerjee@math.buruniv.ac.in} and  $Mahendranath~ Paul^{\dag}$\footnote{mahendrabktpp@gmail.com}}
\myplace{$\dag$Department of Mathematics, The University of Burdwan, Purba Burdwan -713104, W.B., India.} {}
\begin{abstract}
The main object of this paper is to study the concept of weak $I^K$-convergence, a generalization of weak $I^*$-convergence of sequences in a normed space, introducing the idea of weak* $I^K$-convergence of sequences of functionals where $I,K$ are two ideals on $\mathbb{N}$, the set of all positive integers. Also we have studied the ideas of weak $I^K$ and weak* $I^K$-limit points to investigate the properties in the same space.  
\end{abstract}
\myclassification{}
\section{Introduction}
The idea of statistical convergence, an extended form of ordinary convergence, based on the concept of natural density of subsets of $\mathbb{N}$, was introduced independently by Steinhaus\cite{28} and by Fast\cite{17}. Over the years and under different forms of statistical convergence turned out to be one of the most active research areas in the summability theory after the remarkable works of Friday\cite{18,19} and salat\cite{23,27}. In $2000$, Cannor et al\cite{12} introduced the  idea of weak statistical convergence which has been used to characterize Banach spaces with seperable duals. Last few years some basic properties of this concept were studied by many authors in \cite{2,25}. Recently the concept of weak* statistical convergence of sequence of functionals has been given by I. Bala \cite{1}.\\
In 2001, Kostyrko et al \cite{21} extended the idea of statistical convergence into $I$ and $I^*$-convergence which depends on the structure of the ideals $I$ of $\mathbb{N}$. The mutual relation between $I$ and $I^*$-convergence was given in \cite{22} using the condition AP of the ideals.(such ideals are often called P-ideals\cite{3}). Later many works on ideal convergence have been done in \cite{4,5,6,9,13,15}. In 2010, Pehlivan et al\cite{26} introduced the idea of weak $I$ and $I^*$-convergence  in a normed space and using the condition AP, they established a relation between such types convergence. In 2012, Bhardwaj et al \cite{2.1} extended the idea of weak* statistical convergence to weak* ideal convergence of sequence of functionals and shew that the ideas of weak ideal convergence and weak* ideal convergence are identical in a reflexive Banach space.\\
In 2010, M. Macaj et al \cite{24} gave the idea of $I^K$-convergence which is common generalization of all types of $I^*$-convergence. They have shown that if the ideal $I$ has additive property with respect to an another ideal $K$ (i.e. if condition AP$(I,K)$ holds) then $I$-convergence implies $I^K$-convergence. Very recently more results and applications of $I^K$-convergence have been carried out \cite{24.1,11,12}. It seems therefore reasonable to think if we extend the idea of weak and weak* convergence using double ideals in a normed space and in that case we intend to investigate how far several basic properties are affected.\\
In our paper, we start by giving the idea of weak $I^K$-convergence of sequences which is a common generalization of weak $I^*$-convergent sequences as defined in\cite{26} and we present an interrelation between weak $I$ and weak $I^K$-convergence using the condition AP($I,K$). Next we introduce the concept of weak* $I^K$-convergence for a sequence of functionals and observe that the ideas of weak and weak* $I^K$-convergence of sequences of functionals are same in a reflexive Banach space. In the last section of this paper we give the notion of weak $I^K$-limit point of sequences and weak* $I^K$-limit points of sequences of functionals. Since the importance of the notion of weak and weak* convergence in functional analysis is very significant, we realize that the ideas of weak and weak* $I^K$-convergence in a normed space give more general frame for functional analysis to study summability theory as well.
\section{Basic definitions and notations}
Throughout the paper, we use $\mathbb{N}$ to denote the set of all positive integers and $X$ for a normed linear space and $X^*$ for dual of $X$. First recall that a subset $A$ of $\mathbb{N}$ is said to have natural density $d(A)$ if $d(A)=\displaystyle{\lim_{n}}\frac{1}{n}\sum_{k=1}^{n}\chi_A(k)$, provided the limit exists where $\chi_A$ is characteristic function of $A\subset\mathbb{N}$.
\begin{definition}\cite{17}
A sequence $\{x_n\}$ in $X$ is said to be statistically convergent to $l$ if for every $\epsilon>0$ the set $K(\epsilon)=\{k\in\mathbb N:||x_k-l||\geq \epsilon\}$ has natural density zero.
\end{definition}
\begin{definition}\cite{12}
Let $X$ be a normed linear space then a sequence $\{x_n\}_{n\in\mathbb{N}}$ in $X$ is said to be weak statistically convergent to $x\in X$ provided that for any $f\in X^*$, the sequence $\{f(x_n-x)\}_{n\in\mathbb{N}}$ is statistically convergent to $0$. In this case we write $w$-st-$\displaystyle{\lim_{n\rightarrow\infty}}{x_n}=x$.
\end{definition}
\begin{definition}
Let $S$ be a non empty set and a class $I\subset 2^S$ of subset of $S$ is said to an ideal if\\
(i)$A,B\in I$ implies $A\cup B\in I$ and 
(ii)$A\in I, B\subset A$ implies $B\in I$. 
\end{definition}
$I$ is said nontrivial ideal if $S\notin I$ and $ I\neq \{\phi\} $. In view of condition (ii) $\phi\in I $. If $I \subsetneqq 2^S$ we say that $I$ is proper ideal on $S$. A nontrivial ideal $I$ is said admissible if it contains all the singletons of $S$. A nontrivial ideal $I$ is said non-admissible if it is not admissible.
\begin{definition}
Let $F$ be a class of subsets of non-empty set $S$. Then $F$ is said to be a filter in $S$ if \\
(i) $\phi\notin F$,
(ii)$A,B\in F$ implies $A\cap B\in F$ and
(iii)$A\in F, A\subset B$ implies $B\in F$.
\end{definition}
If $I$ is a non-trivial on a non-void set $S$ then $F=F(I)=\{A\subset S:S\setminus A \in I \}$ is clearly a filter on $S$ and conversely. Again $F(I)$ is said associated filter with respect to ideal $I$.
\begin{definition}\cite{21}
A sequence $\{x_n\}_{n\in \mathbb{N}}$ in $X$ is said to be $I$-convergent to $x$ if for any $\epsilon>0$ the set $A(\epsilon)=\{n\in\mathbb{N}:||x_n-x||\geq \epsilon\}\in I$. In this case we write $I-\displaystyle{\lim_{n\rightarrow\infty}}{x_n}=x$.
\end{definition}
\begin{definition}\cite{26}
A sequence $\{x_n\}_{n\in \mathbb{N}}$ in $X$ is said to be weak $I$-convergent to $x\in X$ if for any $\epsilon>0$ and for any $f\in X^*$ the set $A(f,\epsilon)=\{n\in\mathbb{N}:|f(x_n)-f(x)|\geq \epsilon\}\in I$. In this case we write $w-I-\displaystyle{\lim_{n\rightarrow\infty}}{x_n}=x$.
\end{definition}

\begin{note}
It is easy to observe that weak $I$-limit of a weak $I$-convergent sequence is unique and moreover for an admissible ideal $I$, weak convergence implies weak $I$-convergence with the same limit point but converse part is not true which has been shown in paper\cite{26} by an interesting examples.
\end{note}
\begin{note}\label{1}
It is obvious that if two ideals $I_1,I_2$ on $\mathbb{N}$ such that $I_1\subseteq I_2$ then for a sequence $\{x_n\}$ $w$-$I_1$-$\lim{x_n}=x$ implies $w$-$I_2$-$\lim{x_n}=x$.
\end{note}
\section{Weak $I^K$-convergence}
We have already mentioned that our aim to generalize the notion of weak $I^*$-convergence of sequences. We need to modify this definition introduced in\cite{26}.
\begin{definition}(cf.\cite{26})
A sequence $\{x_n\}_{n\in \mathbb{N}}$ in $X$ is said to be weak $I^*$-convergent to $x\in X$ if there exists a set $M\in F(I)$ such that the sequence $\{y_n\}_{n\in \mathbb{N}}\subset X$ defined by
\[y_n=\left\{\begin {array}{ll}
        x_n & \mbox{if $n\in M$} \\
		x & \mbox{if $n\notin M$}
		\end{array}
		\right. \]
is weak-convergent to $x$. we denote it by the notation $w$-$I^*$-$\lim x_n=x$.
\end{definition} 
\begin{definition} (cf.\cite{24})
Let $I,K$ be two ideals on the set $\mathbb{N}$. A sequence $\{x_n\}_{n\in \mathbb{N}}$ in $X$ is said to be weak $I^K$-convergent to $x\in X$ if there exists a set $M\in F(I)$ such that the sequence $\{y_n\}_{n\in \mathbb{N}}\subset X$ defined by
\[y_n=\left\{\begin {array}{ll}
        x_n & \mbox{if $n\in M$} \\
		x & \mbox{if $n\notin M$}
		\end{array}
		\right. \]
is weak $K$-convergent to $x$. we denote it by the notation $w$-$I^K$-$\lim x_n=x$.
\end{definition} 

\begin{remark}
We can give an equivalent definition of weak-$I^K$-convergence in the following way: if there exists an $M\in F(I)$ such that the sequence $\{x_n\}_{n\in M}$ is weak-$K|_M$-convergent to $x$ where $K|_M=\{B\cap M:B\in K\}$.
\end{remark}
\begin{lemma}\label{0}
If $I$ and $K$ are ideals on $\mathbb{N}$, the set of all positive integers and if $\{x_n\}_{n\in\mathbb{N}}$ is a sequence in $X$ such that $w$-$K$-$\lim\{x_n\}=x$, then $w$-$I^K$-$\lim\{x_n\}=x$.
\end{lemma}
The proof follows from the definition of weak $K$-convergence taking $M=\mathbb{N}\in F(I)$ and $y_n=x_n$.
\begin{proposition}
Let $X$ be normed space and $I$ be an ideal on $\mathbb{N}$. A sequence $\{x_n\}_{n\in \mathbb{N}}\subset X$ is weak $I^I$-convergent to $x$ if and only if it is weak $I$-convergent to $x$.
\end{proposition}
\begin{proof}
Let $\{x_n\}$ be weak $I^I$-convergent to $x$ then there exists an $M\in F(I)$ such that the sequence $\{x_n\}_{n\in M}$ is weak-$I|_M$-convergent to $x$. So there exists $G\in F(I)$ such that $\{n\in \mathbb{N}:|f(x_n)-f(x)|<\epsilon\}\cap M=G\cap M$. Clearly $G\cap M\in F(I)$ and $\{n\in \mathbb{N}:|f(x_n)-f(x)|<\epsilon\}\supseteq G\cap M$. Therefore $\{n\in \mathbb{N}:|f(x_n)-f(x)|<\epsilon\}\in F(I)$ i.e. $\{x_n\}$ is weak $I$-convergent to $x$.\\
Converse part follows from lemma \ref{0} taking $K=I$.
\end{proof}
\begin{proposition}
Let $X$ be a normed space and $I,I_1,K$ and $K_1$ be ideals on $\mathbb{N}$ such that $I\subseteq I_1$ and $K\subseteq K_1$. Then for any sequence $\{x_n\}_{n\in\mathbb{N}}$, we have\\
(i) $w$-$I^K$-$\lim x_n=x ~~ \Rightarrow ~~ w$-$I_1^K$-$\lim x_n=x$ and\\
(ii) $w$-$I^K$-$\lim x_n=x ~~ \Rightarrow ~~ w$-$I^{K_1}$-$\lim x_n=x$.
\end{proposition}
\begin{proof}
(i) Now as $w$-$I^K$-$\lim x_n=x$ so there exists an $M\in F(I)$ such that the sequence $\{x_n\}_{n\in M}$ is weak-$K|_M$-convergent to $x$ where $K|_M=\{B\cap M:B\in K\}$. Here $M\in F(I)\subseteq F(I_1)$ as $I\subseteq I_1$. So obviously $w$-$I_1^K$-$\lim x_n=x$.\\
(ii) Again $w$-$I^K$-$\lim x_n=x$ then there exists a set $M\in F(I)$ such that the $\{y_n\}\in X$ given by  
\[y_n=\left\{\begin {array}{ll}
        x_n & \mbox{if $n\in M$} \\
		x & \mbox{if $n\notin M$}
		\end{array}
		\right. \]
is weak $K$-convergent to $x$. Since $K\subseteq K_1$ and from the note $\ref{1}$ we get $\{y_n\}$ is weak $K$-convergent to $x$. Hence $w$-$I^{K_1}$-$\lim x_n=x$.
\end{proof}
\begin{theorem}\label{ab}
Let $I$ and $K$ be ideals on $\mathbb{N}$ and $\{x_n\}_{n\in\mathbb{N}}$ be a sequence in $X$ then \\
 (i) $w$-$I$-$\lim x_n=x ~~ \Rightarrow ~~ w$-$I^K$-$\lim x_n=x$ if $I\subseteq K$.
(ii) $w$-$I^K$-$\lim x_n=x ~~ \Rightarrow ~~ w$-$I$-$\lim x_n=x$ if $K\subseteq I$.
\end{theorem}
\begin{proof}
(i) Since $\{x_n\}$ is weak $I$-convergent to $x\in X$ then for any $\epsilon>0$ and $f\in X^*$ the set $A(f,\epsilon)=\{n\in\mathbb{N}:|f(x_n)-f(x)|\geq \epsilon\}\in I$. Again $I\subseteq K$ so $A(f,\epsilon)\in K$. Therefore the sequence $\{x_n\}$ is weak $K$-convergent to $x$. So from the lemma \ref{0} we get $\{x_n\}$ is weak $I^K$-convergent to $x$.

(ii) Now $w$-$I^K$-$\lim x_n=x$ then there exists a set $M\in F(I)$ such that the sequence $\{y_n\}$ given by  
\[y_n=\left\{\begin {array}{ll}
        x_n & \mbox{if $n\in M$} \\ 
		x & \mbox{if $n\notin M$}
		\end{array}
		\right. \]
is weak $K$-convergent to $x$. So $A(f,\epsilon)=\{n\in\mathbb{N}:|f(y_n)-f(x)|\geq \epsilon\}=\{n\in\mathbb{N}:|f(x_n)-f(x)|\geq \epsilon\}\cap M\in K\subseteq I$. Consequently $\{n\in\mathbb{N}:|f(x_n)-f(x)|\geq \epsilon\}\subseteq (\mathbb{N}\setminus M)\cup A(f,\epsilon)\in I$. So $w$-$I$-$\lim x_n=x$.
\end{proof}
\begin{note}
If  $K\not\subset I$ and $I\not\subset K$ then none of these implications in theorem\ref{ab} may not be true. To support this we cite an example which is weak-$I^K$-convergence but not weak-$I$-convergence.
\end{note} 
\begin{example}
Let $I$ and $K$ be two ideals on $\mathbb{N}$ such that $K\not\subset I$ and $I\not\subset K$, but $K\cap I\neq \phi$. Let $x,y\in X$,$x\neq y$ and consider a set $M\in K\setminus I$. Let us now consider a sequence $\{x_n\}$ with define by
\[x_n=\left\{\begin {array}{ll}
        x & \mbox{if $n\notin M$} \\
		y & \mbox{if $n\in M$}
		\end{array}
		\right. \]
then for every $\epsilon>0$ and $f\in X^*$ we get $\{n\in\mathbb{N}:|f(x_n)-f(x)|\geq \epsilon\}\subset M\in K$. So $w-K-\lim x_n=x$. But, since $x-y\neq \theta$ so from Hanh Banach theorem there exist a $f\in X^*$ such that $f(x-y)=||x-y||$. Choose an $\epsilon=\frac{||x-y||}{2}$. then $\{n\in\mathbb{N}:|f(x_n)-f(x)|\geq \epsilon\}$=$\{n\in M:|f(x_n)-f(x)|\geq \epsilon\}\cup \{n\in\mathbb{N}\setminus M:|f(x_n)-f(x)|\geq \epsilon\}=\{n\in M:|f(y)-f(x)|\geq \frac{||x-y||}{2}\}= \{n\in M:||x-y||\geq \frac{||x-y||}{2}\}M\notin I$. So $w-I-\lim x_n\neq x$.
\end{example}
\begin{note}
Consider any two ideals  $I$ and $K$ on $\mathbb{N}$ then we can construct a new ideal $I\vee K=\{A\cup B:A\in I, B\in K\}$ containing both $I,K$.The dual filter of $I\vee K$ is $F(I\vee K)=\{G\cap H:G\in F(I),H\in F(K)\}$, when $I\vee K$ is non-trivial. It should be noted that if $I\vee K$ is non-trivial ideal and $I,K$ are proper subsets of $I\vee K$ then both $I$ and $K$ are non-trivial. But converse part may or may not be true always. To establish this, following examples are given.
\end{note}
\begin{example}
Let the two sets $P=\{5p:p\in \mathbb{N}\}$ and $S=\{5s-1:s\in \mathbb{N}\}$ now it is clear that $2^{P}$, $2^{S}$ and $2^P\vee 2^S$ all ideals are non-trivial on $\mathbb{N}$.
\end{example}

\begin{example}
Now let  $P$ be set of all odd integers and $S$ be set of all even integers. Then $I=2^{P}$, $K=2^{S}$ both are non-trivial on the whole set $\mathbb{N}$ but $I\vee K$ is not a non-trivial ideal on $\mathbb{N}$.
\end{example}
\begin{theorem}\label{a}
If $I\vee K$ is non-trivial ideal on $\mathbb{N}$ and $X$ is normed space then weak $I^K$-limit of a sequence $\{x_n\}_{n\in \mathbb{N}}$ in $X$ is unique.
\end{theorem}
\begin{proof}
If possible let sequence $\{x_n\}_{n\in \mathbb{N}}$ has two distinct weak $I^K$-limits say $x$ and $y$. Since $x\neq y$ i.e. $(x-y)\neq \theta$ then by a consequence of Hahn Banach theorem there exists $f$ such that $f(x-y)=||x-y||\neq\theta$ then $f(x)\neq f(y)$ and let $\epsilon=\frac{|f(x)-f(y)|}{3}>0$. Since $\{x_n\}_{n\in \mathbb{N}}$ has weak $I^K$-limit $x$ then there exists a set $A_1\in F(I)$ such that the $\{y_n\}\in X$ given by  
\[y_n=\left\{\begin {array}{ll}
        x_n & \mbox{if $n\in A_1$} \\
		x & \mbox{if $n\notin A_1$}
		\end{array}
		\right. \]
is weak $K$-convergent to $x$. So,$\{n\in \mathbb{N}:|f(y_n)-f(x)|\geq \epsilon\}\in K$ i.e. $\{n\in \mathbb{N}:|f(y_n)-f(x)|< \epsilon\}\in F(K)$ which implies that $\{n\in A_1:|f(y_n)-f(x)|< \epsilon\}\cup\{n\in \mathbb{N}\setminus A_1:|f(y_n)-f(x)|< \epsilon\}\in F(K)$ i.e. $(\mathbb{N}\setminus A_1)\cup\{n\in A_1:|f(y_n)-f(x)|< \epsilon\}\in F(K)$ i.e. $\mathbb{N}\setminus(A_1\setminus\{n\in A_1:|f(y_n)-f(x)|< \epsilon\})\in F(K)$ so $A_1\setminus B_1\in K$ where $B_1=\{n\in A_1:|f(x_n)-f(x)|< \epsilon\}$. Similarly as $\{x_n\}$ has weak $I^K$-limit $y$, so there exists a set $A_2\in F(I)$ such that $A_2\setminus B_2\in K$ where $B_2=\{n\in A_2:|f(x_n)-f(y)|< \epsilon\}$.
So, $(A_1\setminus B_1)\cup (A_2\setminus B_2)\in K$ then $(A_1\cap A_2)\cap(B_1\cap B_2)^c\subset (A_1\cap B^c_1)\cup(A_2\cap B^c_2)\in K$. Thus $(A_1\cap A_2)\cap(B_1\cap B_2)^c\in K$ i.e. $(A_1\cap A_2)\setminus(B_1\cap B_2)\in K$. Now by our construction we get $B_1\cap B_2=\phi$. For if $B_1\cap B_2\neq \phi$, let $n\in B_1\cap B_2$ then $|f(x_n)-f(x)|<\epsilon$ and $|f(x_n)-f(y)|<\epsilon$. Therefore, $3\epsilon=|f(x)-f(y)|\leq|f(x)-f(x_n)|+|f(x_n)-f(y)|<2\epsilon$, which is a contradiction. So $A_1\cap A_2\in K$ i.e. $\mathbb{N}\setminus (A_1\cap A_2)\in F(K)\longrightarrow(i)$. Since $A_1,A_2\in F(I)$ so $A_1\cap A_2\in F(I)\longrightarrow (ii)$.
Since $I\vee K$ is non-trivial so the dual filter $F(I\vee K)$ exits.
Now from (i) and (ii) we get $\phi\in F(I\vee K)$, which is a contradiction. Hence the weak $I^K$-limit is unique.
\end{proof}

\begin{theorem}
Let $X$ be normed space and $I,K$ be two ideals on $\mathbb{N}$. A sequence $\{x_n\}_{n\in \mathbb{N}}\in X$ is weak $I^K$-convergent to $x$ if and only if it is weak $(I\vee K)^K$-convergent to $x$.
\end{theorem}
\begin{proof}
Suppose that $\{x_n\}$ is weak $I^K$-convergent to $x$ then there exists an $M\in F(I)$ such that the sequence $\{x_n\}_{n\in M}$ is weak-$K|_M$-convergent to $x$. Since $M\in F(I)$ so it is clear that $M\in F(I\vee K)$. Therefore $\{x_n\}$ is also weak $(I\vee K)^K$-convergent to $x$.\\
Conversely, let $\{x_n\}$ is weak $(I\vee K)^K$-convergent to $x$ then there exists an $M\in F(I\vee K)$ such that the sequence $\{x_n\}_{n\in M}$ is weak-$K|_M$-convergent to $x$. So for any $\epsilon(>0)$ and for every $f\in X^*$ there exists $G\in F(K)$ such that $A(f,\epsilon)\cap M=G\cap M$ where $A(f,\epsilon)={n\in \mathbb{N}:|f(x_n)-f(x)|<\epsilon}$. Since $M\in F(I\vee K)$ then $M=M_1\cap M_2$ for some $M_1\in F(I)$ and $M_2\in F(K)$. Now we have $A(f,\epsilon)\cap M_1\supseteq A(f,\epsilon)\cap M= (G\cap M_2)\cap M_1$. Since $G\cap M_2\in F(K)$, this shows that $A(f,\epsilon)\cap M_1\in F(K|_{M_1})$ i.e. $\{x_n\}$ is weak $I^K$-convergent to $x$.
\end{proof}

In the rest of this section, using additive property of ideals we will investigate  the relationship between weak-$I$ and $I^K$-convergence. Now we recall the definition of $K$-pseudo intersection and then AP$(I,K)$-condition.
\begin{definition} \cite{23}
Let $K$ be an ideal on $\mathbb{N}$. We denote $A\subset_K B$ whenever $A\setminus B\in K.$ If $A\subset_K B$ and $B\subset_K A$ then we denote $A\sim_K B$. Clearly $A\sim_K B \Leftrightarrow A\bigtriangleup B\in K$.\\
If $A\subset_K A_n$ holds for each $n\in \mathbb{N}$ then we say that a set $A$ is $K$-pseudo intersection of a system $\{A_n: n\in \mathbb{N}\}$.
\end{definition}
\begin{definition} \cite{23}\label{2}
Let $I,K$ be ideals on the set $X$. We say that $I$ has additive property with respect to $K$ or that the condition AP$(I,K)$ holds if any one of the equivalent condition of following holds:
(a) For every sequence $(A_n)_{n\in \mathbb{N}}$ of sets from  $I$ there is $A\in I$ such that $A_n\subset_K A$ for every $n'$s.\\
(b) Any sequence $(F_n)_{n\in \mathbb{N}}$ of sets from $F(I)$ has $K$-pseudo intersection in $F(I)$.\\
(c) For every sequence $(A_n)_{n\in \mathbb{N}}$ of sets from the ideal $I$ there exists a sequence $(B_n)_{n\in \mathbb{N}}\subset I$ such that $A_j\sim_K B_j$ for $j\in \mathbb{N}$ and $B=\displaystyle{\cup_{j\in \mathbb{N}}}B_j\in I$.\\
(d) For every sequence of mutually disjoint sets $(A_n)_{n\in \mathbb{N}}\subset I$ there exists a sequence $(B_n)_{n\in \mathbb{N}}\subset I$ such that $A_j\sim_K B_j$ for $j\in \mathbb{N}$ and $B=\displaystyle{\cup_{j\in \mathbb{N}}}B_j\in I$.\\
(e) For every non-decreasing sequence $A_1\subseteq A_2\subseteq \cdots \subseteq A_n\cdots $ of sets from $I$ $\exists$  a sequence $(B_n)_{n\in \mathbb{N}}\subset I$ such that $A_j\sim_K B_j$ for $j\in \mathbb{N}$ and $B=\displaystyle{\cup_{j\in \mathbb{N}}}B_j\in I$.\\
(f) In the Boolean algebra $2^S/K$ the ideal $I$ corresponds to a $\sigma$-directed subset,i.e. every countable subset has an upper bound.
\end{definition}
Note that the proof that the condition (a)to (f) in the definition \ref{2} are equivalent has been given in \cite{23}[lemma 3.9]. Above definition is reformulation of definition given below:
\begin{definition}\cite{16}
Let $I,K$ be ideals on the non-empty set $S$. We say that $I$ has additive property with respect to $K$ or that the condition AP$(I,K)$ holds if for every sequence of pairwise disjoint sets $A_n\in I$, there exists a sequence $B_n\in I$ such that $A_n\bigtriangleup B_n\in K$ for each $n$ and $\displaystyle \cup_{n\in \mathbb N}B_n\in I$ 
\end{definition}
\begin{theorem}
If the condition AP$(I,K)$ holds then weak-$I$-convergence implies weak-$I^K$-convergence, where $I,K$ are two ideals on $\mathbb{N}$.
\end{theorem}
\begin{proof}
Let $\{x_n\}$ be weak $I$-convergent sequence to $x\in X$. Let $f\in X^*$ and choose a sequence of rationals $\{\epsilon_i:i\in \mathbb{N}\}$ so that $\{(f(x)-\epsilon_i,f(x)+\epsilon_i):i\in \mathbb{N}\}$ be a countable base for $\mathbb{R}$ at the point $f(x)$. By weak $I$-convergence of $\{x_n\}$ we have $B_i=\{n:|f(x_n)-f(x)|<\epsilon_i\}\in F(I)$ for each $i$, thus by definition \ref{2}(b) there exists a set $A\in F(I)$ with $A\subset_K B_i$ i.e. $A\setminus B_i\in K$ for all $i$'s. Now it suffices to show that the sequence $\{y_n\}\in X$ given by
\[y_n=\left\{\begin {array}{ll}
        x_n & \mbox{if $n\in A$} \\
		x & \mbox{if $n\notin A$}
		\end{array}
		\right. \]
is weak $K$-convergent to $x$. Now $\{n\in \mathbb{N}:|f(y_n)-f(x)|<\epsilon_i\}=\{n\in A:|f(y_n)-f(x)|<\epsilon_i\}\cup\{n\in \mathbb{N}\setminus A:|f(y_n)-f(x)|<\epsilon_i\}=(\mathbb{N}\setminus A)\cup \{n\in A:|f(x_n)-f(x)|<\epsilon_i\}=(\mathbb{N}\setminus A)\cup(B_i \cap A)=\mathbb{N}\setminus(A\setminus B_i)$. As $A\setminus B_i\in K$ then $\mathbb{N}\setminus(A\setminus B_i)\in F(K)$. Thus $\{n\in \mathbb{N}:|f(y_n)-f(x)|<\epsilon_i\}\in F(K)$ for each $i$ and every $f\in X^*$. Thus $\{y_n\}$ is weak $K$-convergent to $x$. Hence $\{x_n\}$ is weak $I^K$-convergent to $x$.
\end{proof}
\begin{theorem}
Let $I,K$ be ideals on $\mathbb{N}$. If for any sequence $\{x_n\}_{n\in \mathbb{N}}$ in $X$ weak $I$-convergence implies weak $I^K$-convergence then the condition AP$(I,K)$ holds.
\end{theorem}
\begin{proof}
Let $\{y_n\}$ be a sequence in $X$ which is weak $I$-convergent to $x$, Since $X$ is first countable and $f(x)$ is not isolated point in $\mathbb{R}$ then there exists a sequence $\{z_n\}$ of points from $X\setminus \{x\}$ which weak convergent to $x$. Let $\{A_n:n\in \mathbb{N}\}$ be a system of mutually disjoint sets from $I$. Let us define a sequence $\{x_n\}$ as
\[x_n=\left\{\begin {array}{ll}
        z_j & \mbox{if $n\in A_j$} \\
		y_n & \mbox{if $n\notin \cup A_j$}
		\end{array}
		\right. \]
Let $f\in X^*$ be arbitrary. Now $\{n\in \mathbb{N}:|f(x_n)-f(x)|\geq\epsilon\}\subset \{n\in \mathbb{N}:|f(y_n)-f(x)|\geq\epsilon\}\cup \displaystyle{\cup_{j=1}^n}A_j$ implies $\{n\in \mathbb{N}:|f(x_n)-f(x)|\geq\epsilon\}\in I$. This shows that $\{x_n\}$ is weak $I$-convergent to $x$. By our assumption this implies $\{x_n\}$ is weak $I^K$-convergent to $x$ i.e. there exists a set $M\in F(I)$ such that $\{x_n\}_{n\in M}$ is weak $K|_M$-convergent to $x$ i.e. $\{n\in \mathbb{N}:|f(x_n)-f(x)|\geq \epsilon\}\cap M=A\cap M$ for some $A\in K$. This implies that $\{n\in \mathbb{N}:|f(x_n)-f(x)|\geq \epsilon\}\cap M\in K$. Let us define $B_i=A_i\setminus M$ we have $\displaystyle{\cup_{i\in \mathbb{N}}}B_i\subseteq \mathbb{N}\setminus M\in I$. At the same time, for the set $B_i\bigtriangleup A_i=A_i\cap M$ we have $A_i\cap M\subseteq \{n\in \mathbb{N}:|f(x_n)-f(x)|\geq \epsilon\}\cap M$ for any $\epsilon>0$. Consequently $B_i\bigtriangleup A_i\in K$. Hence the condition AP$(I,K)$ holds.
\end{proof}
\section{weak* $I^K$-convergence}
In this section, Following Bala\cite{1} and Bhardwaj et al\cite{2.1}, now we introduce the concept of weak* $I^K$-convergence of sequence of functionals and present some result.
\begin{definition}\cite{2.1}
A sequence $\{f_n\}_{n\in \mathbb{N}}$ in $X^*$ is said to be weak* $I$-convergent to $f\in X^*$ if for any $\epsilon>0$ and for each $x\in X$ the set $A(x,\epsilon)=\{n\in\mathbb{N}:|f_n(x)-f(x)|\geq \epsilon\}\in I$. In this case we write $w^*$-$I$-$\displaystyle{\lim_{n\rightarrow\infty}}{f_n}=f$.
\end{definition}

\begin{definition}\label{3}
A sequence $\{f_n\}_{n\in \mathbb{N}}$ in $X^*$ is said to be weak* $I^*$-convergent to $f\in X^*$ if there exists a set $M=\{m_1<m_2<...<m_k<...\}\in F(I)$ such that $\displaystyle{\lim_{k\rightarrow \infty}}f_{m_k}(x)=f(x)$ for each $x\in X$. In this case we write $w^*$-$I^*$-$\displaystyle{\lim_{n\rightarrow\infty}}{f_n}=f$.
\end{definition}
\begin{theorem}
Let $X$ be a normed space and $\{f_n\}_{n\in \mathbb{N}}$ be a sequence in $X^*$. If $\{f_n\}$ is weak* $I^*$-convergent to $f\in X^*$ then it is weak* $I$-convergent to $f$.
\end{theorem}
\begin{proof}
By assumption, there exists a set $H\in I$ such that for $M=\mathbb{N}\setminus H=\{m_1<m_2<...<m_k<...\}$ we have  $\displaystyle{\lim_{k\rightarrow \infty}}f_{m_k}(x)=f(x)$ for each $x\in X$. Now let $\epsilon>0$ and for this there exists an $N(\epsilon,x)\in \mathbb{N}$ such that $|f_{m_k}(x)-f(x)|< \epsilon$ for each $k> N(\epsilon,x)$. Then we have $\{n\in\mathbb{N}:|f_n(x)-f(x)|\geq \epsilon\}\subset H\cup {m_1,m_2,...,m_{N(\epsilon,x)}}$. Since $I$ is an admissible ideal so right-hand side of the above relation belongs to $I$. Hence the result. 
\end{proof}
\begin{remark}
We can reformulate the definition\ref{3} in the following way: if there exists a set $M\in F(I)$ such that the sequence $\{g_n\}\in X^*$ given by
\[g_n=\left\{\begin {array}{ll}
        f_n & \mbox{if $n\in M$} \\
		f & \mbox{if $n\notin M$}
		\end{array}
		\right. \]
is weak* convergent to  $f$.
\end{remark}
\begin{definition}
 Let $X$ be a normed space with a separable dual $X^*$ and $I,K$ be two ideals on $\mathbb{N}$. A sequence $\{f_n\}_{n\in \mathbb{N}}$ in $X^*$ is said to be weak* $I^K$-convergent to $f\in X^*$ if there exists a set $M\in F(I)$ such that the sequence $\{g_n\}\in X^*$ given by
\[g_n=\left\{\begin {array}{ll}
        f_n & \mbox{if $n\in M$} \\
		f & \mbox{if $n\notin M$}
		\end{array}
		\right. \]
is weak* $K$-convergent to  $f$ and we write $w^*$-$I^K$-$\displaystyle{\lim_{n\rightarrow\infty}}{f_n}=f$
\end{definition}
\begin{theorem}
If $I\vee K$ is a non-trivial ideal on $\mathbb{N}$ and $X$ is normed space with dual $X^*$ then weak* $I^K$-limit of a sequence $\{f_n\}_{n\in \mathbb{N}}$ in $X^*$ is unique.
\end{theorem}
The proof is parallel to proof of theorem\ref{a} with slight modification.
\begin{theorem}\label{c}
Let $X$ be a normed space. If a sequence $\{f_k\}$ in $X^*$ is weak $I^K$-convergent to $f\in X^*$ then it is weak* $I^K$-convergent. 
\end{theorem}
\begin{proof}
By our assumption, $w-I^K-\lim f_k=f$ then there exists a set $M\in F(I)$ such that the sequence $\{g_k\}\in X^*$ given by
\[g_k=\left\{\begin {array}{ll}
        f_k & \mbox{if $k\in M$} \\
		f & \mbox{if $k\notin M$}
		\end{array}
		\right. \]
is weak $K$-convergent to  $f$. Then for every $h\in X^{**}$ and $\epsilon>0$, we have $\{k:|h(g_k)-h(f)|\geq \epsilon\}\in K$. Let $x\in X$ and $F_x=C(x)$ where $C:X\rightarrow X^{**}$ is the canonical mapping we have $F_x(g_k)=g_k(x)$ and $F_x(f)=f(x)$ for every $x\in X$. So in particular for each $x\in X$,$\{k:|F_x(g_k)-F_x(f)|\geq \epsilon\}\in K$ i.e. $\{k:|g_k(x)-f(x)|\geq \epsilon\}\in K$. So the sequence $\{g_k\}$ is weak*
$K$-convergent to $f$. Hence the result.
\end{proof}
\begin{theorem}\label{d}
Let $X$ be a reflexive normed space with dual $X^{*}$. If a sequence $\{f_k\}$ in $X^*$ is weak* $I^K$-convergent to $f\in X^*$ then it is weak $I^K$-convergent to $f$. 
\end{theorem}
\begin{proof}
By our assumption, $w*-I^K-\lim f_k=f$. So there exists a set $M\in F(I)$ such that the sequence $\{g_k\}\in X^*$ given by
\[g_k=\left\{\begin {array}{ll}
        f_k & \mbox{if $k\in M$} \\
		f & \mbox{if $k\notin M$}
		\end{array}
		\right. \]
is weak* $K$-convergent to  $f$. Then for  each $x\in X$ and $\epsilon(>0)$ the set $\{k\in \mathbb{N}:|g_k(x)-f(x)|\geq \epsilon\}\in K$. Let $F\in X^{**}$ then $F=C(x_0)$ for some $x_0\in X$ where $C:X\rightarrow X^{**}$ is the canonical mapping. We have in particular $\{k\in \mathbb{N}:|g_k(x_0)-f(x_0)|\geq \epsilon\}\in K$ since$F(g_k)=g_k(x_0)$ and $F(f)=f(x_0)$. We have $\{k\in \mathbb{N}:|F(g_k)-F(f)|\geq \epsilon\}\in K$ for each $\epsilon(>0)$ and $F\in X^{**}$. So the sequence $\{g_k\}$ is weak
$K$-convergent to $f$. Hence the result.
\end{proof}
\section{weak and weak* $I^K$-limit points}
In this last part, we introduce weak and weak* $I^K$-limit points of sequences and sequence of functionals respectively. First we define weak $I$-limit point of a sequence.
\begin{definition}(cf. \cite{19})
Let $X$ be a normed space and a sequence $\{x_n\}$ be a sequence in $X$. Then $y\in X$ is called an weak $I$-limit point of $\{x_n\}$ if there exists a set $M\notin I$ such that the sequence $\{y_n\}_{n\in \mathbb{N}}\in X$ defined by
\[y_n=\left\{\begin {array}{ll}
        x_n & \mbox{if $n\in M$} \\
		y & \mbox{if $n\notin M$}
		\end{array}
		\right. \]
is weak-convergent to $y$. 
\end{definition} 
\begin{definition}
Let $X$ be a normed space and $I,K$ be two ideals on $\mathbb{N}$. Then $y\in X$ is called an weak $I^K$-limit point of a sequence $\{x_n\}$ if there exists a set $M\notin I,K$ such that the sequence $\{y_n\}_{n\in \mathbb{N}}\in X$ defined by
\[y_n=\left\{\begin {array}{ll}
        x_n & \mbox{if $n\in M$} \\
		y & \mbox{if $n\notin M$}
		\end{array}
		\right. \]
is weak $K$-convergent to $y$. 
\end{definition}
We denote $I(L_w)$ and $I^K(L_w)$ the collection of all weak $I$ and weak $I^K$-limit points of $x_n\in X$.
\begin{theorem}\label{b}
If $K$ is an admissible ideal and $K\subset I$ then $I(L_w)\subset I^K(L_w)$. 
\end{theorem}
\begin{proof}
Let $y\in I(L_w)$, so there exists a set $M\notin I$ such that the sequence $\{y_n\}$ give by
\[y_n=\left\{\begin {array}{ll}
        x_n & \mbox{if $n\in M$} \\
		y & \mbox{if $n\notin M$}
		\end{array}
		\right. \]
is weak-convergent to $y$. Then the sequence of scalars $\{f(y_n)\}$ converges to $f(y)$ for all $f\in X^{*}$ i.e. $\{n:|f(y_n)-f(y)|\geq \epsilon\}$ is a finite set. So $\{n:|f(y_n)-f(y)|\geq \epsilon\}\in K$ as $K$ is an admissible ideal. Therefore $\{y_n\}$ is weak $K$-convergent sequence. Again $M\notin I$ and $K\subset I$, so $M\notin
I,K$. Thus $y$ is weak $I^K$-limit point of $x_n$. Hence the theorem.
\end{proof}
In the similar way we can set the definition of weak* $I^K$-limit points for the sequence of functionals.
\begin{definition}
Let $X$ be a normed space with its dual $X^*$ and $\{f_n\}$ be a sequence in $X^*$. Then $h\in X^*$ is called an weak* $I$-limit point of $\{f_n\}$ if there exists a set $M\notin I$ such that the sequence $\{g_n\}_{n\in \mathbb{N}}\in X^*$ defined by
\[g_n=\left\{\begin {array}{ll}
        f_n & \mbox{if $n\in M$} \\
		h & \mbox{if $n\notin M$}
		\end{array}
		\right. \]
is weak*-convergent to $h$. 
\end{definition} 
\begin{definition}
Let $X$ be a normed space with its dual $X^*$ and $I,K$ be two ideals on $\mathbb{N}$. Then $h\in X^*$ is called an weak* $I^K$-limit point of $\{f_n\}\subset X^*$ if there exists a set $M\notin I,K$ such that the sequence $\{g_n\}_{n\in \mathbb{N}}\in X^*$ defined by
\[g_n=\left\{\begin {array}{ll}
        f_n & \mbox{if $n\in M$} \\
		h & \mbox{if $n\notin M$}
		\end{array}
		\right. \]
is weak* $K$-convergent to $h$.
\end{definition}
We denote $I(L_w*)$ and $I^K(L_w*)$ the collection of all weak* $I$ and $I^K$-limit points of the sequence ${f_n}\in X^{*}$.
\begin{theorem}
If $K$ is an admissible ideal and $K\subset I$ then $I(L_w*)\subset I^K(L_w*)$. 
\end{theorem}
The proof is parallel to proof of the theorem \ref{b}.
\begin{theorem}
Let $X$ be a normed space with its dual $X^*$ . If $h\in X^{*}$ be weak $I^K$-limit point of a sequence $\{f_n\}\subset X^{*}$ then $h$ is also weak* $I^K$-limit point.
\end{theorem}
\begin{proof}
Let $y$ be weak $I^K$-limit point of $\{f_n\}\in X^{*}$ then there exists a set $M\notin I$ such that the sequence $\{g_n\}_{n\in \mathbb{N}}\in X^{*}$ defined by
\[g_n=\left\{\begin {array}{ll}
        f_n & \mbox{if $n\in M$} \\
		h & \mbox{if $n\notin M$}
		\end{array}
		\right. \]
is weak $K$-convergent to $y$. Again by theorem \ref{c} we get $\{g_n\}$ is weak* $K$-convergent to $h$. Hence $h$ is weak* $I^K$-limit point.
\end{proof}
\begin{remark}
By the theorem \ref{d} we get weak* $K$-convergence implies weak $K$-convergence when $X$ is reflexive normed space. Therefore converse of above theorem holds when $X$ is a reflexive normed space.
\end{remark}

\section{Acknowledgment}
The second author is grateful to the University of Burdwan, W.B., India for providing him State Fund Fellowship during the preparation of this work.


\end{document}